\documentclass{article}
\usepackage{amsmath,amssymb,amsfonts,theorem}
\usepackage{graphicx,epsfig,color} %
\usepackage{xspace,algorithmic}
\graphicspath{{./Figures/}}

\newcommand{\subscr}[2]{#1_{\textup{#2}}}

\newcommand{\setdef}[2]{\{#1 \, : \; #2\}}

\newcommand{\ba}{\begin{array}}
\newcommand{\ea}{\end{array}}

\newcommand{\be}{\begin{equation}}
\newcommand{\ee}{\end{equation}}
\newcommand{\eps}{\varepsilon}

\renewcommand{\l}{\left}
\renewcommand{\r}{\right}

\def\1{\mathbf{1}}

\renewcommand{\natural}{\mathbb{N}}

\newcommand{\integernonnegative}{\mathbb{Z}_{\ge0}}

\newcommand{\Exp}{\mathbb{E}}
\renewcommand{\Pr}{\mathbb{P}}

\renewcommand{\deg}{d}
\newcommand{\outdeg}[1]{{\deg^+_{#1}}}
\newcommand{\indeg}[1]{\deg^-_{#1}}
\newcommand{\outdegmax}{{\subscr{\deg^+}{max}}}
\newcommand{\indegmax}{{\subscr{\deg^-}{max}}}
\newcommand{\degmax}{\subscr{\deg}{max}}

\newcommand{\V}{\mathcal{V}} 
\newcommand{\E}{\mathcal{E}}
\newcommand{\G}{\mathcal{G}}

\newcommand{\F}{\mathcal{F}}	

\def\neigh{\mathcal{N}}

\newcommand{\card}[1]{|#1|}
\newcommand{\xave}{\subscr{x}{ave}}

\newcommand{\CR}{C_{R}}
\newcommand{\Var}{\operatorname{Var}}

\newtheorem{definition}{Definition}[section]
\newtheorem{theorem}{Theorem}[section]
\newtheorem{corollary}[theorem]{Corollary}
\newtheorem{lemma}[theorem]{Lemma}
\newtheorem{proposition}[theorem]{Proposition}
{ \theorembodyfont{\normalfont} 
\newtheorem{remark}[theorem]{Remark}

}

\def\QEDopen{{\setlength{\fboxsep}{0pt}\setlength{\fboxrule}{0.2pt}\fbox{\rule[0pt]{0pt}{1.3ex}\rule[0pt]{1.3ex}{0pt}}}}
\def\QED{\QEDopen}

\title{The asymptotical error of broadcast gossip averaging algorithms}

\author{ Fabio Fagnani \and Paolo Frasca\thanks{Fabio Fagnani and Paolo Frasca are with the Dipartimento
di Matematica, Politecnico di Torino, Corso Duca degli Abruzzi 24, 10129 Torino, Italy.
{\tt\small \{paolo.frasca,fabio.fagnani\}@polito.it}}%
}
\date{\today}

\begin{document}
\maketitle

\begin{abstract}                       
In problems of estimation and control which involve a network, efficient distributed computation of averages is a key issue. This paper presents theoretical and simulation results about the accumulation of errors during the computation of averages by means of iterative ``broadcast gossip'' algorithms. Using martingale theory, we prove that the expectation of the accumulated error can be bounded from above by a quantity which only depends on the mixing parameter of the algorithm and on few properties of the network: its size, its maximum degree and its spectral gap. Both analytical results and computer simulations show that in several network topologies of applicative interest the accumulated error goes to zero as the size of the network grows large.
\end{abstract}



\section{Introduction}
Distributed computation of averages is an important building block to solve problems of estimation and control over networks. As a reliable time-independent communication topology may be unlikely in the applications, a growing interest has been devoted to randomized ``gossip'' algorithms to compute averages. In such algorithms, at each time step, a random subset of the nodes communicates and performs computations. Unfortunately, some of these iterative algorithms do not deterministically ensure that the average is preserved through iterations, and due to the accumulation of errors, in general there is no guarantee that the typical algorithm realization will converge to a value which is close to the desired average.

In the present paper we focus on one of these randomized algorithm, the Broadcast Gossip Algorithm (BGA). In this algorithm, a node is randomly selected at each time step to broadcast its current value to its neighbors, which in turn update their values by a local averaging rule. Since these updates are not symmetric, it is clear that the average is not preserved, but instead is changed at each time step by some amount. In this paper we study how these errors accumulate, and how much the convergence value of the algorithm differs from the original average to be computed.

\subsection{Contribution}
In this paper, we study the bias, or asymptotical error, committed by a Broadcasting Gossip averaging algorithm, and we show that large neighborhoods and a large mixing parameter induce a large asymptotical error. As a theoretical contribution, we study the average of states as a martingale, and by this interpretation we prove that on symmetric graphs the expectation of the accumulated error can not be larger than a constant times $\displaystyle\frac{q}{1-q}\frac{\degmax^2}{N\lambda_1},$ where $q$ is the ``mixing'' parameter of the algorithm, $N$ is the network size, $\degmax$ is the maximum degree of the nodes, and $\lambda_1$ is the network spectral gap. For some families of graphs (e.g, expander graphs), this is enough to prove that the bias goes to zero as $N$ goes to infinity. 
Further, by means of simulations we show that, on some example graph topologies, the mean bias is an increasing function of the mixing parameter and is proportional, on large networks, to the ratio between degree and size of the network. In particular, whenever $\degmax=o(N),$ the simulated bias goes to zero as $N$ goes to infinity. 

\subsection{Related works}
The paper~\cite{FF-SZ:08a} provides a general theory for randomized linear averaging algorithms, and presents a few example algorithms, some of which do not preserve the average of the states. 
Among these algorithms, the Broadcast Gossip Algorithm, studied in the present paper, has been attracting a wide interest, for its natural application to wireless networks: main references include the paper~\cite{TCA-MEY-ADS-AS:09} and the recent survey~\cite{AGD-SK-JMF-MGR-AS:10}. While it is simple to give conditions to ensure that the expectation of the convergence value is equal to the initial average, the problem of estimating the difference between expectation and realizations is harder, and has received partial answers in a few papers. 
In~\cite{TCA-ADS-AGD:09} the authors study a related communication model, in which the broadcasted values may not be received with a probability which depends on the transmitter and receiver nodes, and claim that ``aggressive updating combined with large neighborhoods [\ldots] result in more variance [of the convergence value] within the short time to convergence''. This intuition extends to the Broadcast Gossip Algorithm which we are considering in this paper.
Actually, in~\cite{TCA-MEY-ADS-AS:09} the variance of the limit value has been estimated for general graphs, with an upper bound which is proportional to $\l(1-\frac{\lambda_1}{\lambda_{N-1}} \frac1{1-\frac12\frac qN \lambda_{N-1}}\r),$ where $\lambda_i$ is the $i$-th smallest eigenvalue of the graph Laplacian.
This bound, however, is not useful to prove that the bias goes to zero as $N$ grows, a fact which has been proved in~\cite{FF-PF:10} for sequences of Abelian Cayley graphs with bounded degree, using tools from algebra and Markov chain theory.
Analogous problems can be studied for other randomized algorithms which do not preserve the average. For instance, in~\cite{FF-SZ:08b} two related algorithms are studied, in which node values are sent to one random neighbor only. If at each time step one random node sends its value, then the variance has an upper bound which is proportional to $\frac{q}{1-q} \frac1N$, while if at each time step every node sends its value, then the bound is proportional to $\frac{q^2}{\lambda_1 N}.$
%


%

\section{Problem statement}
Let a graph $\G=(\V,\E)$ with $\E\subset\V\times\V$ be given, together with $N=\card{\V}$ real numbers $\{x_v\}_{v\in\V}\subset[0,L]$.  For every node $v\in\V,$ we denote its out-neighborhood by $\neigh^+_v=\setdef{u\in\V}{(u,v)\in\E}$, and its in-neighborhood by $\neigh^-_v=\setdef{u\in\V}{(v,u)\in\E}$.
The following Broadcasting Gossip Algorithm (BGA) is run in order to estimate the average $\xave=N^{-1}\sum_{v\in\V}x_v$. 

At each time step $t\in\integernonnegative$, one node $v$ is sampled from a uniform distribution over $\V$. Then, node $v$ broadcasts its state $x_v(t)$ to its neighbors $u\in \neigh^+_v,$ which in turn update their states as
\be\label{eq:update-step}
x_u(t+1)=(1-q)x_u(t)+q x_v(t).
\ee
The parameter $q\in(0,1)$ is said to be the {\em mixing parameter} of the algorithm.
If instead $u\not\in\neigh^+_v$, there is no update: $x_u(t+1)=x_u(t).$

It is known from~\cite[Corollary~3.2]{FF-SZ:08a} that the BGA converges, in the sense that there exists a random variable $x^*$ such that almost surely $\lim_{t\to+\infty}x(t)=x^*\1.$ 
Let now $\xave(t)=N^{-1}\sum_{v\in\V}x_v(t)$.
Although one can find conditions to ensure that $\Exp[x^*]=\xave(0)$, in general $x^*$ is not equal to $\xave(0)$. Then, it is worth to ask how far the convergence value is from the initial average. To study this bias in the computation of the average, we define $$\beta(t)=|\xave(t)-\xave(0)|^2.$$
The goal of this work is to study this quantity, and in particular its limit $\Exp[\beta(\infty)]:=\lim_{t\to\infty}\Exp[\beta(t)]$, with a special attention to its dependence on the size of the network. In particular, we shall say that the algorithm is {\em asymptotically unbiased} if 
$\displaystyle\lim_{N\to+\infty}\Exp[\beta(\infty)]=0.$

\section{Analysis}
\subsection{A simplistic bound}
Using~\eqref{eq:update-step}, it is immediate to compute that, given $v$ to be the broadcasting node at time $t$,
\be\label{eq:increm}
\xave(t+1)-\xave(t)= \frac{q}N \sum_{u\in\neigh^+_v} (x_v(t)-x_u(t)).
\ee
Then, we can obtain the following deterministic bound on the error introduced at each time step, 
\be\label{eq:increm-bound}
|\xave(t+1)-\xave(t)|\le\frac{q}N \outdeg{v} L \le \frac{q \outdegmax }N L,\ee
where $\outdeg{v}=\card{\neigh^+_v}$ is the out-degree of node $v$, and $\outdegmax=\max_{v\in\V}{\outdeg{v}}.$ 
This simple bound is worth some informal remarks. Indeed, Equation~\eqref{eq:increm-bound} suggests that choosing a low value of the mixing parameter $q$, and a graph with low degree $\outdegmax$ and large size $N$, may ensure a small bias in the computation of the average. However, by choosing $q$, $N$ or $\outdegmax$, one affects the speed of convergence of the algorithm. Assume one is interested in an accurate computation, and chooses low values for $q$ and $\outdegmax$, compared to $N$. This choice would likely imply a slow convergence, and in turn a slow convergence may enforce to run the algorithms for a larger number of steps, in order to meet the same precision requirement. These extra steps, however, would introduce extra errors, thus possibly wasting the desired advantage in the accuracy.
We argue from this discussion that there is a delicate {\em trade-off between speed and accuracy} for the BGA algorithm. The results presented in the next sections will shed light on this trade-off.

\subsection{The average as a martingale}
In this section, we shall derive a general bound on $\Exp[\beta(\infty)]$ in terms of the topology of the graph. The derivation is based on applying the theory of martingales to the stochastic processes $x(t)$ and $\xave(t)$. The reader can find the essentials of martingale theory in~\cite{JJ-PP:03} or in~\cite{ANS:89}.

\begin{definition}\label{def:martingale}
Given a sequence (filtration) of $\sigma$-algebras $\{\sigma_n\}_{n\in\integernonnegative},$ a sequence of random variables $\{M_n\}_{n\in\integernonnegative}$ is a $\sigma_n$-martingale if $\Exp[M_m|\sigma_n]=M_n$, for any $m\ge n$.
\end{definition}

Our first result states that $\xave(t)$ is a martingale with respect to the filtration induced by $x(t)$. Before the statement, we need some definitions. Let $\outdeg{v}=\card{\neigh^+_v}$ and $\indeg{v}=\card{\neigh^-_v}$ be the out-degree and in-degree of node $v$. The graph $\G$ is said to be {\em balanced} if $\indeg{u}=\outdeg{u}$ for all $u\in \V$.
Given a set of random variables $X$, we denote by $\sigma(X)$ the sigma-algebra generated by the random variables in $X$.

\begin{proposition}\label{prop:xave-martingale}
Let us consider the BGA algorithm and the filtration $\F_t=\sigma(\setdef{x(s)}{s\le t}).$
If the graph $\G$ is balanced, then the sequence of random variables $\{\xave(t)\}_{t\in\integernonnegative}$ is a square-integrable $\F_t$-martingale.
\end{proposition}
\begin{proof}
First, note that $\xave(t)$ is $\F_t$-measurable. Moreover, Equation~\eqref{eq:increm} implies that for all $t\ge 0$,
\begin{align*}
\Exp[\xave(t+1)-\xave(t) | \F_t]
 =&\, \frac{1}N \sum_{v\in\V} \l(\frac{q}N \sum_{u\in\neigh^+_v} (x_v(t)-x_u(t)) \r)
\\ =&\, \frac{q}{N^2} \l[ \sum_{v\in\V} \outdeg{v} x_v(t) - \sum_{v\in\V} \sum_{u\in\neigh^+_v} x_u(t) \r]
\\ =&\, \frac{q}{N^2} \l[ \sum_{v\in\V} \outdeg{v} x_v(t) - \sum_{u\in\V} \indeg{u} x_u(t) \r]
\\ =&\, \frac{q}{N^2} \l[ \sum_{v\in\V} \l(\outdeg{v}-\indeg{v}\r) x_v(t) \r]=0,
\end{align*}
since we are assuming $\indeg{u}=\outdeg{u}$ for every $u\in\V$.
Then, the sequence of random variables $(\xave(t))_{t\in\integernonnegative}$ is an $\F_t$-martingale. 
Moreover, the fact that $\xave(t)\in [0,L]$ implies that the martingale is bounded in $L^p$ for every $p\ge 1$, and in particular square-integrable.
\end{proof}

\bigskip
Let us define the distance from the agreement as $$d(t):=\frac1N \sum_{v\in\V} (x_v(t)-\xave(t))^2.$$ Using this definition, we can prove an inequality which is a refinement of~\eqref{eq:increm-bound}. Let $\degmax^\pm=\max_{v}\{\deg^\pm_v\}$ and $\degmax=\max\{\indegmax,\outdegmax\}.$
\begin{lemma}\label{lem:increm-var}
Let $\G$ be balanced. Then, the increments of the martingale $\{\xave(t)\}_{t\in\integernonnegative}$ have bounded variance, in particular, for all $t\ge 0$,
\be\label{eq:increm-var}
\Exp[(\xave(t+1)-\xave(t))^2 | \F_t]\le 4 \frac{q^2\degmax^2}{N^2} d(t).
\ee
\end{lemma}
\begin{proof}
By~\eqref{eq:increm}, we have
\begin{align*}
\Exp[(\xave(t+1)-\xave(t))^2 | \F_t] =&\, 
\frac{1}N \sum_{v\in\V} \l(\frac{q}N \sum_{u\in\neigh^+_v}(x_v(t)-x_u(t)) \r)^2\\
\le&
\frac{1}N \sum_{v\in\V} \frac{q^2}{N^2} \outdeg{v} \sum_{u\in\neigh^+_v}(x_v(t)-x_u(t))^2\\
\le& \frac1N \sum_{v\in\V} \frac{q^2}{N^2} 2 \Big( (\outdeg{v})^2 (x_v(t)-\xave(t))^2 
+ \outdeg{v} \sum_{u\in\neigh^+_v}(\xave(t)-x_u(t) )^2 \Big)\\
\le& \,2 \l(\frac{q^2}{N^2} (\outdegmax)^2 d(t)+ \outdegmax\indegmax d(t) \r)\\
\le& \,4\, \frac{q^2 \degmax^2}{N^2} d(t).
\end{align*}
This completes the proof.
\end{proof}

\bigskip
We define the rate of convergence of the algorithm as
\begin{equation*}
    R:=\sup_{x(0)}\limsup_{t\rightarrow +\infty} \Exp[d(t)]^{1/t}.
\end{equation*}
Then,  there exists a positive constant $\CR$, depending on $x(0)$, such that 
$\Exp[d(t)] \le \CR R^t. $
This fact, combined with Lemma~\ref{lem:increm-var}, implies that 
\begin{align}
\nonumber\Exp[(\xave(t+1)-\xave(t))^2] =&\, \Exp\l[\Exp[(\xave(t+1)-\xave(t))^2 | \F_t]\r]\\
\le &\,4\, \CR \frac{q^2 \degmax^2}{N^2} R^t
\label{eq:bound-on-d}
\end{align}

\bigskip
We recall that the {\em spectral gap} of the graph $\G$ is the smallest (in modulus) non-zero eigenvalue of its Laplacian matrix, and we denote this quantity by $\lambda_1$.
It is well-known that $\lambda_1$ relates to the mixing rate of Markov chains, and to the speed of convergence of gossip algorithms~\cite{SB-AG-BP-DS:06,FF-SZ:08a}: the larger the spectral gap, the faster the convergence. 
In particular, let us assume that the graph $\G$ be {\em symmetric}, that is, such that $\neigh_{v}^+=\neigh_{v}^-$ for all $v\in\V.$ Then, we know from~\cite[Equation~(18)]{FF-PF:10} that
\begin{equation}\label{eq:bound-on-R}
R\le 1-\frac{2q(1-q)}{N}\lambda_1.
\end{equation}
Using these facts, we are going to prove the next result about the asymptotic behavior of the bias as $t\to+\infty$. 
\begin{proposition}\label{prop:beta-lambda}
If $\G$ is symmetric, then there exists a constant $C>0$ such that 
$$
\Exp[\beta(\infty)] \le C \frac{q}{1-q} \frac{\degmax^2}{N \lambda_1}.
$$
\end{proposition}
\begin{proof}
Using the orthogonality of the increments of square-integrable martingales, we observe that
\begin{align*}
\lim_{t\to+\infty} \Exp[(\xave(t+1)-\xave(0))^2]
=&
\lim_{T\to+\infty} \Exp\l[\l( \sum_{t=0}^{T-1}(\xave(t+1)-\xave(t))\r)^2\r]\\
=&\lim_{T\to+\infty}  \Exp\l[  \sum_{t=0}^{T-1}(\xave(t+1)-\xave(t))^2 
\right. \\ & \left.  \qquad
+ 2 \sum_{t=1}^{T-1}\sum_{s<t}  (\xave(t+1)-\xave(t)) (\xave(s+1)-\xave(s)) \r]\\
=&\lim_{T\to+\infty}  \Big[ \sum_{t=0}^{T-1}  \Exp\l[(\xave(t+1)-\xave(t))^2\r] 
\\ &  \qquad 
+ 2 \sum_{t=1}^{T-1}\sum_{s<t} \Exp [(\xave(t+1)-\xave(t))
(\xave(s+1)-\xave(s)) ] \Big]\\
=& \lim_{T\to+\infty} \sum_{t=0}^{T-1}  \Exp\l[(\xave(t+1)-\xave(t))^2\r].
\end{align*}
By applying Equation~\eqref{eq:bound-on-d}
\begin{align*}
\lim_{t\to+\infty} \Exp[(\xave(t+1)-\xave(0))^2]\le & \,4\,\CR \frac{q^2 \degmax^2}{N^2} \lim_{T\to+\infty} \sum_{t=0}^{T-1} R^t 
= \,4\, \CR \frac{q^2 \degmax^2}{N^2} \frac{1}{1-R}.
\end{align*}
The inequality in~\eqref{eq:bound-on-R}  implies that $\displaystyle\lim_{t\to+\infty} \Exp[(\xave(t+1)-\xave(0))^2]\le 2 \,\CR \frac{q}{1-q} \frac{\degmax^2}{N \lambda_1}$. The thesis then follows, with $C=2\CR$, by applying the dominated convergence theorem. 
\end{proof}

\bigskip

Note that the proof of Proposition~\ref{prop:beta-lambda} also implies that 
$$\sup_{t} \Exp[(\xave(t+1)-\xave(0))^2]\le C \frac{q}{1-q} \frac{\degmax^2}{N \lambda_1}.$$
On the other hand, for a convergent square-integrable martingale $M_t$,  we know by Doob's maximal inequality that $\Exp{[\sup_t M_t^2]}\le 4\, \Exp[\lim_t M_t^2].$ Then, we can immediately obtain the following finite-time counterpart of Proposition~\ref{prop:beta-lambda}.

%

\begin{theorem}\label{th:xave-finite-time}
If $\G$ is symmetric, then there exists $C'>0$ such that
$$ \Exp\l[\sup_{t\in\natural}\beta(t)\r]\le C' \frac{q}{1-q} \frac{\degmax^{\,2}}{N \lambda_1}.$$
\end{theorem}

\smallskip
Let us now consider a sequence of graphs $\G_N$ of increasing size $N$. In such a sequence, both $\degmax$ and $\lambda_1$ are functions of $N$. In this context, Proposition~\ref{prop:beta-lambda} implies the following corollary.
\begin{corollary}
Let  $ I \subseteq \natural$ and $(\G_N)_{N\in I}$ be a sequence of symmetric graphs such that $\G_N=(\V_N,\E_N)$ and $\card{\V_N}=N$. If $$\frac{\degmax^{\,2}}{\lambda_1}=o(N) \quad \text{as } N\to+\infty,$$
then the BGA algorithm is {\em asymptotically unbiased}.
\end{corollary}
Note that, since $\xave(t)$ converges a.s. and $\lim_{t\to+\infty} \xave(t)\in [0,L]$, it is trivial to find an upper bound on $\Exp[\beta(\infty)]$ which does not depend on $N$. The interest of the above corollary is in giving a sufficient condition for $\Exp[\beta(\infty)]$ to be $o(1)$ as $N\to\infty.$

\begin{remark}\label{rem:concentration}
Applying Markov's inequality, we see from Proposition~\ref{prop:beta-lambda} that for any $c>0$,
$$\Pr[\beta(\infty)>c]\le C \frac{q}{1-q} \frac{\degmax^2}{N \lambda_1} \frac1c $$
In the applications, one is often interested in computing an average because the average is the maximum likelihood estimator of the expectation of a random variable. In such context, the average enjoys the property that the mean square error, committed by approximating the expectation by the average of N samples, is equal to 1/N. For this reason, one would like to ensure that the bias introduced by the Broadcast Gossip algorithm is not larger than $1/N$.  
If we take $c=1/N$, we get
\begin{align*}
\Pr\l[\beta(\infty)>\frac1N\r]\le C \frac{q}{1-q} \frac{\degmax^2}{ \lambda_1} .
\end{align*}
Then, provided the right-hand-side of this inequality does not diverge, we can choose the mixing parameter $q$ so that with a positive given probability the bias is below $1/N$. In such case, if our purpose is distributed estimation of an expected value, averages which are approximated by a Broadcast Gossip Algorithm are as good as averages computed by a centralized method. \hfill\QED
\end{remark}

\section{Examples}\label{sec:Examples}
In this section, we show that the BGA is asymptotically unbiased on several example topologies which have been considered in the literature.
Given a graph $\G_N$ of size $N$, we denote by  $\lambda_1(N)$ and  $\degmax(N)$ its spectral gap and maximum degree, respectively.
We consider the following example sequences of graphs.
\begin{itemize}
\item {\bf Expander graphs.} \\
A sequence of graphs is said to be an {\em expander} sequence if there exist $d\in\natural$ and $c>0$ such that  for every $N$, $\lambda_1(N)\ge c$ and $\degmax(N)\le d.$ In this case, $\frac{\degmax^2}{ \lambda_1}$ is bounded, and this fact implies by Proposition~\ref{prop:beta-lambda} that $$\Exp[\beta(\infty)]=O\l(\frac1N\r) \text{ as } {N\to+\infty},$$ and then the BGA is asymptotically unbiased and Remark~\ref{rem:concentration} applies.
An example of an expander sequence is given by a sequence of {\em de Bruijn graphs} on $n$ symbols of increasing dimensions $k$. 
A de Bruijn graph on $n$ symbols of dimension $k$ has $n^k$ vertices and edges from any $i$ to $ni, ni + 1, ni + 2, \ldots, ni + k - 1$ (all modulo $n^k$). Their expander properties have already been applied to efficient averaging algorithms in~\cite{JCD-RC-SZ:09}.
\smallskip
\item {\bf Hypercube graphs.} \\
The $n$-dimensional hypercube graph is the graph obtained drawing the edges of a $n$-dimensional hypercube. It has $N = 2^n$ nodes which can be identified with the binary words of length $n$, and two nodes are neighbors if the corresponding binary words differ in only one component. For these graphs it is known, for instance from~\cite[Example~7]{PF-RC-FF-SZ:08}, that $\lambda_1(N)=\Theta(1/{\log N})$ and $\degmax(N)=O( \log N).$
Then, $$\Exp[\beta(\infty)]=O\l(\frac{\log^3 N}{N}\r)\text{ as }{N\to+\infty}$$
and the BGA is asymptotically unbiased.
\smallskip
\item {\bf $k$-dimensional square lattices.} \\
We consider square lattices obtained by tiling a $k$-dimensional torus, with $N=n^k$ nodes.
For these graphs we know~\cite[Theorem~6]{RC-FF-AS-SZ:08}, that $\lambda_1(N)=\Theta(1/N^{2/k})$ and $\degmax(N)=2k.$
Then, $$\Exp[\beta(\infty)]=O\l(\frac{1}{N^{1-2/k}}\r) \text{ as }{N\to+\infty},$$
and we argue that $k-$lattices are asymptotically unbiased if $k>3.$
However, we know that the BGA is asymptotically unbiased also if $k=1,2$: this has been proved in~\cite{FF-PF:10} using different techniques.
\smallskip
\item { \bf Random geometric graphs.}\\
 We can also consider random sequences of geometric graphs constructed as follows. We sample $N$ points from a uniform distribution over the unit square, and we let nodes $i$ and $j$ be connected if the two corresponding points in the square are less than $r(N)$ far apart, with $r(N)=1.1\sqrt{\frac{\log(N)}{N}}$.
For these graphs, we know from~\cite{MP:03} that with high probability $\lambda_1(N)=\Theta(1/N)$ and $\degmax(N)=O(\log N).$
Then, with high probability $\Exp[\beta(\infty)]=O\l(\log^2 N\r)\text{ as }{N\to+\infty},$
and we can not conclude asymptotical unbiasedness.
\smallskip
\item {\bf Complete graphs.} \\
For these graphs, $\lambda_1(N)=N$ and $\degmax(N)=N-1.$\\
Then, we can not conclude from Proposition~\ref{prop:beta-lambda} that the BGA is asymptotically unbiased on complete graphs. Actually, in~\cite{FF-SZ:08a} it is shown that the BGA is {\em not} asymptotically unbiased on complete graphs, and in particular 
\be\label{eq:complete-bias}
\Exp[\beta(\infty)]=\Var(x(0)) \frac{q}{2-q} \frac{N-1}{N},
\ee
where by $\Var(x(0))$ we denote the (sample) variance of the initial condition.
\end{itemize}

\section{Simulations}\label{sec:simulations}
We have extensively simulated the evolution of BGA algorithm on sequences of graphs, and in particular on the example topologies presented in Section~\ref{sec:Examples}. 
In this section, we account for our results about the dependence of the bias $\beta(\infty)$ on the size $N$ and on the parameter $q$. 

Our simulation setup is as follows. Let $q$, $N$ and the graph topology be chosen. For every run of the algorithm we generate a vector of initial conditions $x(0)$, sampling from a uniform distribution over $[0,1].$\footnote{If the topology is random, namely a random geometric topology as described above, it is also sampled at this stage. Disconnected realizations are discarded: however, disconnected realizations are very few in our random geometric setting and their number decreases as $N$ grows, so that they are less than 2\% when $N>50$.}
Then, we run the algorithm until the disagreement $d(t)$ is below a small threshold $\eps$, which we set at $10^{-4}$. At this time $T^\eps=\inf\setdef{t\ge 0}{d(t)\le \eps}$, the algorithm is stopped, and $\beta(T^\eps)$ is evaluated. In order to simulate the expectation of $\Exp[\beta(\infty)]$, we average the outcome of 1000 realizations of $\beta(T^\eps)$.
Our results about the dependence on $N$ are summarized in Figure~\ref{fig:bias}, which plots the average bias against $N$ in a log-log diagram. As expected, complete graphs are not asymptotically unbiased, while all other topologies, in which the degree is $o(N),$ are asymptotically unbiased. In particular, for de Bruijn graphs on 2 symbols, ring graphs and torus graphs, the bias is $\Theta(N^{-1})$, whereas for hypercubes and random geometric graphs the bias is $\Theta(\frac{\log(N)}{N}).$
Overall, our set of simulations suggests that
$$\Exp[\beta(\infty)]=\Theta\l(\frac\degmax N\r) \quad \text{as }N\to\infty.$$
Results presented in Figure~\ref{fig:Nq-db} confirm that this asymptotical law is independent of the choice of $q$, provided $q<1$. Note indeed that if we run the BGA with $q=1$ in the update~\eqref{eq:update-step}, the convergence value is always one of the initial values, sampled according to a uniform distribution. This implies that, if $q=1$, then $\Exp[\beta(\infty)]=\Var(x_u(0))=1/12.$ If instead $q\in(0,1)$, simulations in Figure~\ref{fig:q} show that the bias $\Exp[\beta(\infty)]$ is an increasing function of the mixing parameter.

\begin{figure}[htb]\centering
\includegraphics[width=.8\columnwidth]{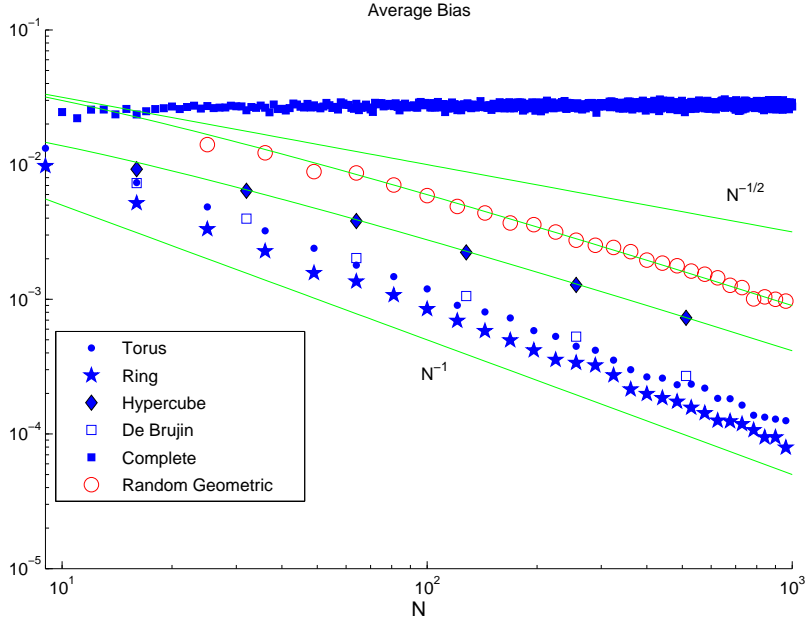} 
\caption{Average bias $\beta(\infty)$ as a function of the graph size $N$, for $q=0.5$ and on different topologies (various marks). Solid lines represent $\Theta(N^{-1/2})$, $\Theta(\log N /N)$, $\Theta(\log{N}/N)$ and $\Theta(N^{-1})$, respectively.}
\label{fig:bias}
\end{figure}
\begin{figure}[htb]\centering
\includegraphics[width=.8\columnwidth]{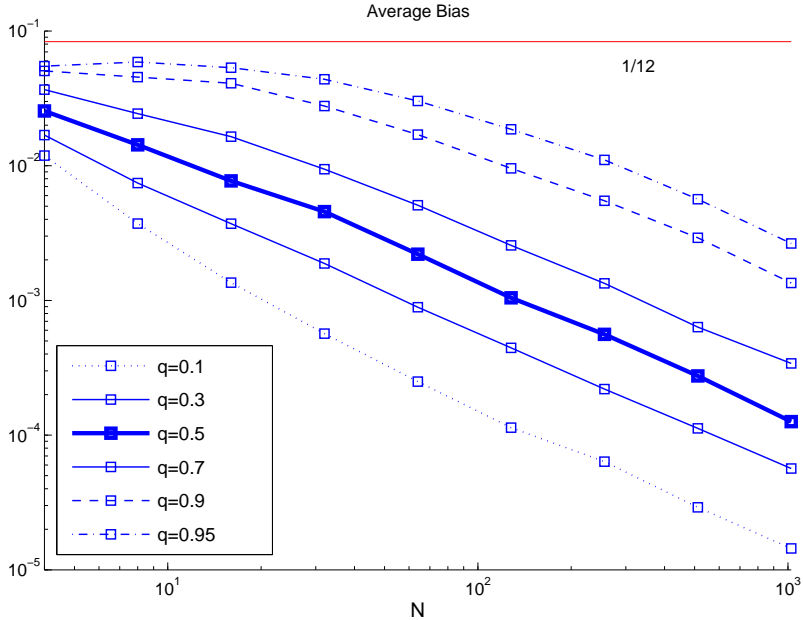} 
\caption{Average bias $\beta(\infty)$ as a function of the size $N$, on a sequence of de Bruijn graphs on $2$ symbols, for different values of $q$. The solid horizontal line represents the theoretical value $\beta(\infty)=1/12$ obtained for $q=1$.}
\label{fig:Nq-db}
\end{figure}
\begin{figure}[htb]\centering
\includegraphics[width=.8\columnwidth]{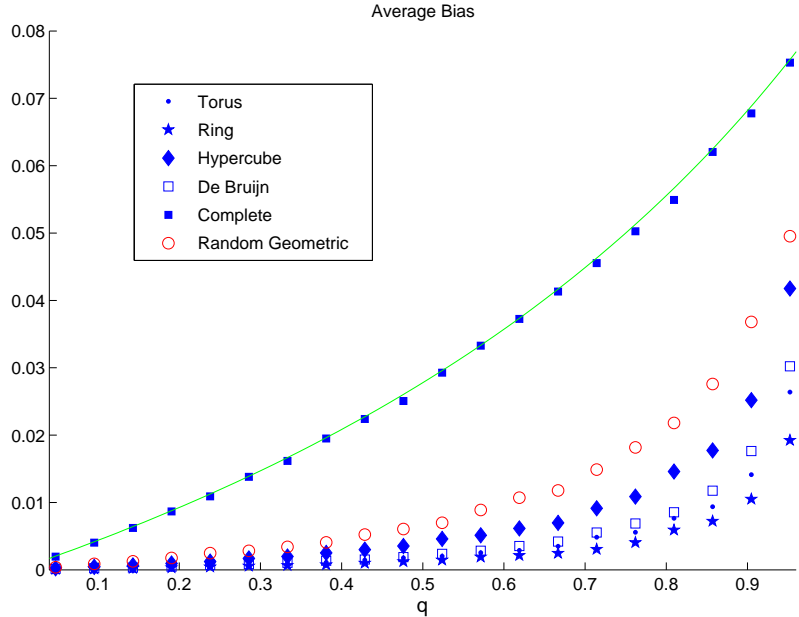} 
\caption{Average bias $\beta(\infty)$ as a function of the mixing parameter $q$, for $N=64$ and on different topologies (various marks). The solid line is~\eqref{eq:complete-bias}, confirming the theory in the complete case. }
\label{fig:q}
\end{figure}


\section{Conclusion}
In this paper, we have proven that on any symmetric graph, $\Exp[\beta(\infty)]= O\l(\frac{\degmax^{\,2}}{N \lambda_1}\r),$ and in particular the BGA is asymptotically unbiased on expander graphs.
On the other hand, simulations suggest that, on sequences of (almost) regular graphs with degree $d(N)$, the bias is such that $\Exp[\beta(\infty)]= \Theta\l(\frac{d(N)} N\r).$  
Our future research will be devoted to find a bound on $\Exp[\beta(\infty)]$ which ensure asymptotic unbiasedness on a wider set of topologies, and more in general to study the trade-offs between speed and accuracy in gossip algorithms.

\section*{Acknowledgements}
P.~Frasca wishes to thank M.~Boccuzzi for his advice in implementing the simulations and S.~Zampieri for finding an error in an earlier draft.

\bibliographystyle{plain}
\bibliography{aliasFrasca,PF,References}

\end{document}